\documentclass[letterpaper,12pt]{amsart}

\usepackage{fullpage}
\usepackage[foot]{amsaddr}

\usepackage{amssymb}
\usepackage{hyperref}
\usepackage[utf8]{inputenc}
\usepackage[T1]{fontenc}
\usepackage{url}
\usepackage{thmtools}
\usepackage{booktabs}
\usepackage{tikz}
\usetikzlibrary{arrows,patterns,shapes}

\let\leq\leqslant
\let\geq\geqslant
\let\le\leqslant
\let\ge\geqslant
\let\epsi\varepsilon
\let\rho\varrho

\newcommand{\brac}[1]{{\left(#1\right)}}
\newcommand{\sbrac}[1]{{\left[#1\right]}}
\newcommand{\tbrac}[1]{{\left<#1\right>}}
\newcommand{\set}[1]{\left\{#1\right\}}
\newcommand{\norm}[1]{{\left|#1\right|}}
\newcommand{\floor}[1]{{\left\lfloor #1 \right\rfloor}}
\newcommand{\ceil}[1]{{\left\lceil #1 \right\rceil}}

\newcommand{\Oh}[1]{O\brac{#1}}
\newcommand{\oh}[1]{o\brac{#1}}

\newcommand{\Nat}{\mathbb{N}}
\newcommand{\PosNat}{\Nat_{+}}
\newcommand{\Real}{\mathbb{R}}
\newcommand{\PosReal}{\Real_{+}}
\newcommand{\Rational}{\mathbb{Q}}
\newcommand{\Z}{\mathbb{Z}}

\newtheorem{theorem}{Theorem}
\newtheorem{corollary}[theorem]{Corollary}
\newtheorem{lemma}[theorem]{Lemma}
\newtheorem{conjecture}[theorem]{Conjecture}
\newtheorem{observation}[theorem]{Observation}
\newtheorem{question}{Question}
\theoremstyle{definition}
\newtheorem{definition}[theorem]{Definition}
\newtheorem{example}[theorem]{Example}
\newtheorem{remark}[theorem]{Remark}

\title{Online Coloring of Short Intervals}

\thanks{%
Joanna Chybowska-Sokół was partially supported by the National Science Center of Poland under grant no.~2016/23/N/ST1/03181.
Grzegorz Gutowski was partially supported by the National Science Center of Poland under grant no.~2016/21/B/ST6/02165.
Patryk Mikos was partially supported by the National Science Center of Poland under grant no.~2014/14/A/ST6/00138.
Adam Polak was partially supported by the Polish Ministry of Science and Higher Education program \emph{Diamentowy Grant} under grant no.~DI2012 018942.
}
\author[J.~Chybowska-Sok\'o\l{}]{Joanna Chybowska-Sok\'o\l{}}
\author[G.~Gutowski]{Grzegorz Gutowski}
\author[K.~Junosza-Szaniawski]{Konstanty Junosza-Szaniawski}
\author[P.~Mikos]{Patryk Mikos}
\author[A.~Polak]{Adam Polak}

\address[G.~Gutowski, P.~Mikos, A.~Polak]{
Institute of Theoretical Computer Science,
Faculty of Mathematics and Computer Science,
Jagiellonian University, Krak\'ow, Poland
}
\email{\{gutowski,mikos,polak\}@tcs.uj.edu.pl}
\address[J.~Chybowska-Sok\'o\l{},K.~Junosza-Szaniawski]{
Faculty of Mathematics and Information Science,
Warsaw University of Technology, Poland
}
\email{\{j.sokol,k.szaniawski\}@mini.pw.edu.pl}	

\begin{document}

\begin{abstract}
We study the online graph coloring problem restricted to the intersection
graphs of intervals with lengths in $[1,\sigma]$.
For $\sigma=1$ it is the class of unit interval graphs, and for $\sigma=\infty$
the class of all interval graphs. Our focus is on intermediary classes.
We present a $(1+\sigma)$-competitive algorithm, which beats the state of the art
for $1 < \sigma < 2$, and proves that the problem we study can be strictly easier
than online coloring of general interval graphs.
On the lower bound side, we prove that no algorithm is better than
$5/3$-competitive for any $\sigma>1$, nor better than
$7/4$-competitive for any $\sigma>2$, and that no algorithm beats the
$5/2$ asymptotic competitive ratio for all, arbitrarily large, values of $\sigma$.
That last result shows that the problem we study can be strictly harder than unit interval
coloring. Our main technical contribution is a recursive composition of strategies,
which seems essential to prove any lower bound higher than $2$.
\end{abstract}

\maketitle

\section{Introduction}

In the \emph{online graph coloring} problem the input graph is presented to the
algorithm vertex by vertex, along with all the edges adjacent to the already
presented vertices. Each vertex must be assigned a color, different than any of
its neighbors, immediately and irrevocably at the moment it is presented,
without any knowledge of the remaining part of the graph. The objective is to
minimize the number of used colors.
The problem and its variants attract much attention, both for theoretical
properties and practical applications in network multiplexing, resource
allocation, and job scheduling.

The standard performance measure, used to analyze online algorithms, is the
\emph{competitive ratio}, i.e., the worst-case guarantee on the ratio of the
solution given by an online algorithm to the optimal offline solution (see
Section~\ref{sec:results} for a formal definition).

In the general case, of online coloring of arbitrary graphs, there is no hope for
any algorithm with a constant competitive ratio. The best known
algorithm~\cite{Halldorsson97} uses $\Oh{\chi\cdot n/\log n}$ colors for
$n$-vertex $\chi$-colorable graphs, i.e.~it is $\Oh{n/\log n}$-competitive,
and there is a lower bound~\cite{HalldorssonS94} showing that no online graph
coloring algorithm can be $\oh{n/\log^2 n}$-competitive.
It is thus common to study the problem restricted to particular graph classes.

Having in mind the applications in scheduling, one of the important special
cases is the class of \emph{interval graphs}, i.e.~intersection graphs of
intervals on the real line.
The classic result is by Kierstead and Trotter~\cite{KiersteadT81},
who designed a $3$-competitive algorithm and proved a matching lower bound.
However, in the special case of \emph{unit interval graphs}, i.e.~intersection
graphs of intervals of a fixed length, already the simple greedy
FirstFit algorithm is $2$-competitive~\cite{EpsteinL05}.


Coloring unit interval graphs can model only a restricted scheduling setting,
with all jobs having the same processing time. On the other hand, allowing
arbitrary processing times, modeled by (general) interval graphs, might be a
too permissive setting, precluding efficient online algorithms. Thus, we ask what
happens in between the interval and unit interval graph classes. In particular,
we are interested in the optimal competitive ratio of online coloring algorithms
for intersection graphs of intervals of length restricted to a fixed range.
Formally, let us introduce the $\sigma$-interval coloring problem.

\begin{definition}
For $\sigma \ge 1$, the \emph{$\sigma$-interval coloring} problem asks:
Given a sequence of closed intervals $[l_1, r_1], [l_2, r_2], \ldots, [l_n, r_n]$,
such that $1 \le (r_i - l_i) \le \sigma$ for every $i \in [n]$,
find a sequence of \emph{colors}, $c_1, c_2, \ldots, c_n$, such that
\[\forall_{i \neq j}\ \big([l_i, r_i] \cap [l_j, r_j] \neq \emptyset\big) \implies
\brac{c_i \neq c_j}\text{,}\]
minimizing the number of distinct colors $\big|\{c_1, c_2, \ldots, c_n\}\big|$.
\end{definition}

We study the problem in the online setting, i.e., intervals are presented one by
one, in an arbitrary order, and each interval has to be colored immediately and
irrevocably after it is presented.

Note that we choose to include the interval representation in the input, instead
of presenting the mere graph. It seems a plausible modeling choice given the
scheduling applications. Moreover, it lets algorithms exploit geometric properties
of the input, and not only structural graph properties.
Naturally, any lower bound obtained for this variant of the problem transfers to
the harder variant with no interval representation in the input.


Among others, we look to answer the following two apparent questions.

\begin{question}
\label{que:easier}
Is $\sigma$-interval coloring strictly easier than interval coloring?
\end{question}

\begin{question}
\label{que:harder}
Is $\sigma$-interval coloring strictly harder than unit interval coloring?
\end{question}

Naturally, we ask these questions in the online setting, where easier (harder)
mean smaller (greater) best possible competitive ratio.

\subsection{Our Results}
\label{sec:results}

Before we state our results, let us give a formal definition of the competitive
ratio. In this paper we focus on the \emph{asymptotic} competitive ratio.

\begin{definition}
Let $A$ be an online graph coloring algorithm, and let $A(\chi)$ denote the
maximum number of colors $A$ uses to color any graph which can be colored offline
using $\chi$-colors (i.e.~its chromatic number is at most $\chi$).
We say that $A$ has the asymptotic competitive ratio $\alpha$ (or that $A$ is
$\alpha$-competitive, for short), if
$\limsup_{\chi\to\infty} \frac{A(\chi)}{\chi} \le \alpha$.
\end{definition}

Another popular performance measure for online algorithms is the \emph{absolute}
competitive ratio, which requires that $\frac{A(\chi)}{\chi} \le \alpha$
holds for \emph{all} $\chi$ (and not only in the limit). The choice of the asymptotic,
instead of absolute, competitive ratio for our analysis makes things easier for
the algorithm and harder for the lower bounds. In our algorithm, sadly, we do
not know how to get rid of a constant additive overhead, which vanishes only
with growing $\chi$. This is in contrast to the FirstFit and Kierstead-Trotter
algorithms, whose claimed competitive ratios are not only asymptotic but also
absolute. The good side is, our lower bounds for the asymptotic competitive ratio
imply the identical lower bounds for the absolute competitive ratio.

\subsubsection*{Algorithm.}
Our positive result is the existence of a $(1+\sigma)$-competitive algorithm.

\begin{restatable}{theorem}{AlgoThm}
\label{thm:algo}
For every $\sigma \in \Rational$, $\sigma \geq 1$,
there is an algorithm for online $\sigma$-interval coloring with $1+\sigma$
asymptotic competitive ratio.
\end{restatable}

Note that for $\sigma' > \sigma$ every $\sigma'$-interval coloring algorithm is
also a correct $\sigma$-interval coloring algorithm, with the same upper bound
on its competitive ratio. Therefore, for $\sigma \in \Real \setminus \Rational$
Theorem~\ref{thm:algo} yields an online $\sigma$-interval coloring algorithm
with a competitive ratio arbitrarily close to $1+\sigma$.
This distinction between rational and irrational values of $\sigma$ becomes 
somewhat less peculiar in the light of the results of Fishburn and
Graham~\cite{Fishburn85}, who proved, among other things, that the classes
$C(\sigma)$'s of graphs with interval representation with lengths in $[1, \sigma]$
are \emph{right-continuous}, i.e.~$C(\sigma) = \bigcap_{\tau > \sigma} C(\tau)$,
exactly at irrational $\sigma$'s.

Until now, the state-of-the art was
the $2$-competitive FirstFit algorithm~\cite{EpsteinL05} for $\sigma=1$, and
the $3$-competitive Kierstead-Trotter algorithm~\cite{KiersteadT81} for $\sigma>1$.
Our algorithm matches the performance of FirstFit for $\sigma=1$,
and beats the Kierstead-Trotter algorithm for $\sigma<2$.

The algorithm is inspired by the recent result for online coloring of unit disk
intersection graphs~\cite{UnitDisk}.
We cover the real line with overlapping blocks, grouped into a constant number
of classes, with each class constituting a partition of the real line.
We assign to each class a private set of available colors.
When an interval is presented, the algorithm chooses, in a round-robin fashion,
a block containing the interval's left end, and greedily picks a color from
the block's class.

While not being overly complicated, the algorithm already answers positively our
Question~\ref{que:easier}. Indeed, since there cannot exist a
less-than-$3$-competitive algorithm for (general) interval
coloring~\cite{KiersteadT81}, $\sigma$-interval coloring is strictly easier,
for $\sigma < 2$.

\subsubsection*{Lower Bounds.}

Our negative results include a series of constructions with the following consequences.

\begin{restatable}{theorem}{FiveThirdsThm}
\label{thm:53}
For every $\sigma > 1$ there is no online algorithm for $\sigma$-interval
coloring with the asymptotic competitive ratio less than $5/3$.
\end{restatable}

\begin{restatable}{theorem}{SevenFourthsThm}
\label{thm:74}
For every $\sigma > 2$ there is no online algorithm for $\sigma$-interval
coloring with the asymptotic competitive ratio less than $7/4$.
\end{restatable}

\begin{restatable}{theorem}{FiveHalvesThm}
\label{thm:52}
For every $\epsi > 0$ there is $\sigma \geq 1$ such that there is no online
algorithm for $\sigma$-interval coloring with the asymptotic competitive ratio
$5/2 - \epsi$.
\end{restatable}

The following, more illustrative, statement is a direct corollary of Theorem~\ref{thm:52}.

\begin{corollary}
There is no online algorithm that works for all $\sigma \geq 1$ and uses at most
$2.499\cdot\omega + f(\sigma)$ colors for $\omega$-colorable graphs
(for any function $f$).
\end{corollary}

Theorem~\ref{thm:52} gives a positive answer to our Question~\ref{que:harder}.
Indeed, while FirstFit is $2$-competitive for unit interval coloring, there is
no $2$-competitive algorithm for $\sigma$-interval coloring (for large enough
$\sigma$), therefore the latter problem is strictly harder.
Working out the exact number from our proof, this starts to be the case for
$\sigma > 2^{78}$, however we did not attempt to optimize the constant.

Our proofs of Theorems~\ref{thm:53} and~\ref{thm:74} can be considered as
generalizations of the $3/2$ lower bound for online coloring of unit interval
graphs by Epstein and Levy~\cite{EpsteinL05}. In particular, we heavily use
their \emph{separation strategy}, which also appears in~\cite{Azar06}.

Our main technical contribution is a recursive composition of strategies,
which seems essential to prove any lower bound higher than $2$.
Our $5/2$ lower bound (Theorem~\ref{thm:52}) borrows also from
the work of Kierstead and Trotter~\cite{KiersteadT81}. However, in order to
control the length of intervals independently of the number of colors, we cannot
simply use the pigeonhole principle, as they did. Instead, we develop
Lemmas~\ref{lem:4sets} and~\ref{lem:4split}, which let us overcome this issue,
at a cost of a worse bound for the competitive ratio, i.e.~$5/2$ instead of $3$.

\subsection{Related Work}

Interval graphs have been intensively studied since the
sixties~\cite{Benzer59,Lekkeikerker62}, and, in particular, they are known to be
\emph{perfect}, i.e.~the chromatic number $\chi$ of an interval graph always
equals the size of the largest clique $\omega$ (see, e.g., \cite{Golumbic04}).
To construct an optimal coloring offline it is sufficient to color the graph greedily
in a nondecreasing order of the left ends of the intervals.

For the most basic approach for online coloring, that is the FirstFit algorithm,
the competitive ratio for interval graphs is still not known exactly.
After a series of papers, the most recent results state that FirstFit is at least
$5$- and at most $8$-competitive~\cite{Kierstead16,Narayanaswamy08}.
Kierstead and Trotter~\cite{KiersteadT81} designed a more involved online
coloring algorithm, which uses at most $3\omega-2$ colors for $\omega$-colorable
interval graphs, and proved that there exists a strategy that forces any online
coloring algorithm to use exactly that number of colors. The same lower and upper
bounds were obtained independently by Chrobak and Ślusarek~\cite{Chrobak88,Slusarek89}.
For intersection graphs of intervals of unit length any online coloring algorithm
uses at least $\frac{3}{2}\omega$ colors, and FirstFit uses at most $2\omega-1$
colors~\cite{EpsteinL05}.

It seems a natural question to ask if it is possible to improve the bound of
$3\omega-2$ by assuming that interval lengths belong to a fixed range.
The study of interval graphs with bounded length representations was initiated
by Fishburn and Graham~\cite{Fishburn85}. However, it focused mainly on
the combinatorial structure, and not its algorithmic applications.

Kierstead and Trotter~\cite{KiersteadT81} give, for every $\omega\in\PosNat$,
a strategy for Presenter to construct an $\omega$-colorable set of intervals
while forcing Algorithm to use at least $3\omega-2$ colors. However, the lengths
of presented intervals increase with the increasing $\omega$. For this reason,
with the interval length restricted to $\sbrac{1,\sigma}$, their lower bound
is only for the absolute competitive ratio and does not exclude, say, an algorithm
that always uses at most $2\omega+\sigma^{10}$ colors.
On the contrary, in Theorem~\ref{thm:52} we rule out the existence of such an algorithm.

\section{Algorithm}

\AlgoThm*

\begin{proof}
Let us present an algorithm which, in principle, works for any real $\sigma$,
however only for a rational $\sigma$ it achieves the declared competitive ratio.
The algorithm has a positive integer parameter $b$. Increasing the parameter
brings the asymptotic competitive ratio closer to $1+\sigma$ at the cost of
increasing the additive constant. More precisely, given an $\omega$-colorable
set of intervals our algorithm colors it using at most
$\ceil{b\cdot(1 + \sigma)} \cdot \brac{\frac{\omega}{b} + b - 1}$
colors, and thus its competitive ratio is
$\frac{\ceil{b\cdot(1 + \sigma)}}{b} + \Oh{1/\omega}$.
For a rational $\sigma$, in order to obtain exactly the declared $1 + \sigma$ asymptotic competitive ratio it is
sufficient to set $b$ to the smallest possible denominator of a simple fraction
representation of $\sigma$.
Let $\varphi = \ceil{b\cdot(1 + \sigma)}$.
The algorithm will use colors from the set $\set{0, 1, \ldots, \varphi-1} \times \Nat$.

Now, let us consider the partition of the real line into \emph{small blocks}.
For $i \in \Z$, the $i$-th small block occupies the interval
$[i\cdot\frac{1}{b}, (i+1)\cdot\frac{1}{b})$.
Moreover, we define \emph{large blocks}. The $i$-th large block occupies the interval
$[i\cdot\frac{1}{b}, i\cdot\frac{1}{b} + 1)$. See Figure~\ref{fig:blocks}.

\begin{figure}
\centering
\includegraphics{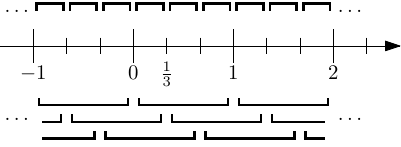}
\caption{Small blocks (above), and large blocks (below), for $b=3$}
\label{fig:blocks}
\end{figure}

Let us point out certain properties of the blocks, which will be useful in the further analysis.
Each large block is the union of $b$ consecutive small blocks, and each small block is a subset of $b$ consecutive large blocks.
Further, the length of a large block is $1$, and for any two intervals of length in $\sbrac{1,\sigma}$ that both have the left endpoint in the same large block, the two intervals intersect.
Thus, the intervals whose left endpoints belong to a fixed large block form a clique.
Finally, if the indices of two large blocks differ by at least $\varphi$, then any two intervals -- one with the left endpoint in one block, the other with the left endpoint in the other -- do not intersect.

With each small block the algorithm associates a \emph{small counter}, and
with each large block the algorithm associates a \emph{large counter}.
Let $S_i$ denote the small counter of the $i$-th small block, and $L_j$ denote the large counter of the $j$-th large block.
Initially, all the small and large counters are set to zero.

The small counter $S_i$ tracks how many intervals with left endpoint in the $i$-th small block appeared so far.
Based on the small counter the algorithm assigns, in a round-robin fashion, the processed interval to one of the $b$ large blocks containing its left endpoint.
The large counter $L_j$ tracks how many intervals were assigned to the $j$-th large block so far.
The color of the processed interval will depend on the large counter.

To assign a color to an interval, the algorithm proceeds as follows:
\begin{enumerate}
\item Let $i$ be the index of the small block containing the left endpoint of the interval.
\item Let $j$ be the index of the large block containing the left endpoint of the interval such that $j \equiv S_i \pmod{b}$. Note that there is exactly one such $j$.\label{item:round_robin}
\item Assign to the interval the color $(j\ \mathrm{mod}\ \varphi, L_j)$.
\item Increase the small counter $S_i$ by one.
\item Increase the large counter $L_j$ by one.
\end{enumerate}

First let us argue that the algorithm outputs a proper coloring.
Consider any two intervals which were assigned the same color.
Let $j_1$ and $j_2$ denote the indices of the large blocks selected for these intervals by the algorithm.
Since the colors of the two intervals have the same first coordinates, we have that $j_1 \equiv j_2 \pmod{\varphi}$.
However, since the second coordinates, which are determined by large counters, are also the same, $j_1$ and $j_2$ must be different, and thus they differ by at least $\varphi$. Hence the left endpoints of the large blocks $j_1$ and $j_2$ are at least $1+\sigma$ apart, and the two considered intervals do not intersect, thus the coloring is proper.

It remains to bound the number of colors in terms of the clique number $\omega$.
Let $j$ be the index of the maximum large counter $L_j$ at the end of the algorithm.
Clearly, the algorithm used at most $\varphi \cdot L_j$ colors in total.
Let $C$ denote the set of intervals with the left endpoints in the $j$-th large block and colored with a color in $\set{j\ \mathrm{mod}\ \varphi} \times \Nat$.
Observe that $\norm{C} = L_j$.
Let $x_k$ denote the number of intervals in $C$ which have the left endpoint in the $k$-th small block.
Recall that the $j$-th large block is the union of $b$ small blocks -- indexed $j$, $j+1$, \ldots, $j+b-1$ -- and thus $L_j = x_j + x_{j+1} + \cdots + x_{j+b-1}$.
Because of the the formula $j \equiv S_i \pmod{b}$ in the step~(\ref{item:round_robin}) of the algorithm, the large counter $L_j$ is incremented by one out of every $b$ intervals with the left endpoints in any given small block.
Hence either $x_k = \lfloor\frac{S_k}{b}\rfloor$ or $x_k = \lfloor\frac{S_k}{b}\rfloor+1$.
In particular
\[
S_k \ge b \cdot (x_k - 1) + 1\text{.}
\]
Let $D$ denote the set of all intervals with the left endpoints in the $j$-th large block.
We can bound the number of intervals in $D$
\[
\norm{D}\ =\ \sum_{k=j}^{j+b-1}S_k\ \ge\ \sum_{k=j}^{j+b-1}\brac{b \cdot (x_k - 1) + 1}\ =\ b \cdot (L_j - b) + b\text{.}
\]
Recall that $D$ is a clique and thus the clique number $\omega$ of the input graph is at least the size of $D$.
Therefore $L_j \le \frac{\omega + b \cdot (b - 1)}{b}$, and the algorithm used
at most 
\[
\ceil{b\cdot(1 + \sigma)} \cdot \brac{\frac{\omega}{b} + b - 1}
\]
colors.
\end{proof}

\section{Lower Bounds}

\subsection{Technical Overview}


\subsubsection*{Algorithm-vs-Presenter game.}

In order to prove lower bounds for online problems, it is often convenient to
look at the problem as a combinatorial game between two players, Algorithm
and Presenter. In our case, in each round Presenter reveals an interval, and
Algorithm immediately and irrevocably assigns a color to it.
Algorithm tries to minimize the number of different colors it assigns.
Contrarily, the Presenter's goal is to force Algorithm to use as many
colors as possible, while guaranteeing that the introduced set of intervals is
colorable with a smaller number of colors, and contains only short intervals.
A strategy for Presenter implies a lower bound on the competitive ratio of any
algorithm solving the problem.

\subsubsection*{Separation Strategy.}

\begin{figure}
\centering
\begin{tikzpicture}

\draw [very thick, color=red    ] (0.0,1.5) -- (4.0,1.5);
\draw [very thick, color=green  ] (0.2,1.2) -- (4.2,1.2);
\draw [very thick, color=blue   ] (0.4,0.9) -- (4.4,0.9);
\draw [very thick, color=cyan   ] (0.6,0.6) -- (4.6,0.6);
\draw [very thick, color=yellow ] (0.8,0.3) -- (4.8,0.3);
\draw [very thick, color=magenta] (1.0,0.0) -- (5.0,0.0);

\draw [very thick, color=green] (5.5,0.0) -- (9.5,0.0);
\draw [very thick, color=red  ] (5.5,0.3) -- (9.5,0.3);
\draw [very thick, color=blue ] (5.5,0.6) -- (9.5,0.6);

\draw (7.5,0) node [anchor=north] {\emph{initial intervals}};
\draw (2.5,0) node [anchor=north] {\emph{separation intervals}};

\end{tikzpicture}
\caption{Separation strategy}
\label{fig:separation}
\end{figure}

Epstein and Levy~\cite{EpsteinL05} prove their $\frac{3}{2}$ lower bound for online
coloring of unit intervals by giving the following strategy for Presenter.
The strategy has three phases.


\begin{enumerate}
\item Presenter introduces a clique of $\frac{\omega}{2}$ equal (i.e.~located at the same place) \emph{initial intervals}. Algorithm has to use $\frac{\omega}{2}$ different colors. Let $\mathcal{X}$ denote the set of these colors.
\item Just to the left of the initial intervals, Presenter introduces a clique of $\omega$ \emph{separation intervals}. These intervals are not equal, they all have slightly different endpoints, and in total they occupy a place of length $1+\epsi$. None of them intersect the initial intervals. The specific way the separation intervals are presented ensures that all separation intervals with colors in $\mathcal{X}$ have their left (resp.~right) endpoints to the left of all left (resp.~right) endpoints of separation intervals with colors not in $\mathcal{X}$ (see Figure~\ref{fig:separation}).
In particular, each of $\frac{\omega}{2}$ right-most separation intervals has a color not in $\mathcal{X}$.
\item Presenter introduces a clique of equal $\frac{\omega}{2}$ \emph{final intervals} that intersect all the initial intervals, and $\frac{\omega}{2}$ right-most separation intervals.
\end{enumerate}
Algorithm has to use at least $\frac{3}{2}\omega$ different colors, but the largest clique size (i.e.~the offline optimum) is only $\omega$.

The second phase of the above strategy is an example of the \emph{separation strategy}.
The formal details are included in the proof of Lemma~\ref{lem:lower_32}.

\subsubsection*{Recursive Composition of Strategies.}
\label{sec:reccomp}

Let us generalize the above strategy.
Observe, that instead of presenting a clique in the first phase, Presenter can use an arbitrary strategy, presenting a set of intervals with the clique number equal to $\frac{\omega}{2}$ but possibly enforcing more than $\frac{\omega}{2}$ colors. Moreover, Presenter might be able to achieve a better trade-off by allowing a clique of size $\beta\omega$ (for some $0 < \beta < 1$) in the first phase, and introducing a clique of size $(1-\beta)\omega$ in the third phase. Actually, it turns out that if the strategy used in the first phase enforces a competitive ratio $\alpha$ (e.g.~presenting a clique has $\alpha=1$), then it is optimal to set $\beta=\frac{1}{1+\alpha}$.

Assume we have a strategy that enforces a competitive ratio $\alpha$,
uses intervals of length at most $\sigma$,
and needs a place of length $M$ on the real axis
(e.g.~presenting a clique gives $\alpha=\sigma=M=1$).
Consider the following strategy.
\begin{enumerate}
  \item Presenter plays the assumed strategy to obtain a set of initial intervals with the clique number $\frac{\omega}{1+\alpha}$. Algorithm has to use $\alpha\frac{\omega}{1+\alpha}$ different colors, denoted by $\mathcal{X}$.
  \item Presenter introduces a clique of $\omega$ separation intervals (of unit length), and at least the right-most $\omega-|\mathcal{X}| \geq \frac{\omega}{1+\alpha}$ of them get new colors, not in $\mathcal{X}$.
  \item Presenter introduces a clique of $\alpha\frac{\omega}{1+\alpha}$ final intervals of length $M+\epsi$ that intersect all the initial intervals, and the $\frac{\omega}{1+\alpha}$ right-most separation intervals.
\end{enumerate}
In total, the intervals can be colored with $\omega$ colors, but the algorithm uses at least
\[\alpha\frac{\omega}{1+\alpha}+ \frac{\omega}{1+\alpha} + \alpha\frac{\omega}{1+\alpha} = \frac{(2\alpha + 1)\omega}{1+\alpha}\]
colors, i.e.~it is at most $(2-\frac{1}{\alpha+1})$-competitive

Let $S_0$ be a trivial strategy that presents a clique of equal unit intervals.
Now, let $S_{i+1}$ be a strategy obtained as above by playing $S_{i}$ in the first phase.
It enforces a competitive ratio $\alpha_{i+1}$, uses intervals of length at most $\sigma_{i+1}$, and needs a place of length $M_{i+1}$, where
$\alpha_0 = \sigma_0 = M_0 = 1$, and each
$\alpha_{i+1} = 2 - \frac{1}{\alpha_{i}+1}$, $\sigma_{i+1} = M_{i}+\epsi$, $M_{i+1} = M_{i} + 1 + \epsi$.
By solving the recurrence equations we get the following corollary.

\begin{corollary}
\label{cor:informal_lower_32}
For every $n \in \PosNat$ and every $\epsi > 0$,
there is no online algorithm for $(n+\epsi)$-interval coloring
with the asymptotic competitive ratio less than $\frac{F_{2n+1}}{F_{2n}}$,
where $F_{n}$ is the $n$-th Fibonacci number ($F_0 = F_1 = 1$, $F_{n+2}=F_{n+1}+F_{n}$).
\end{corollary}

Note that this method cannot give a lower bound higher than
$\lim_{n \rightarrow \infty} \frac{F_{2n+1}}{F_{2n}} = \frac{1+\sqrt{5}}{2} \approx 1.61803$.
However, we can get arbitrarily close to this bound.
That is, for every $\epsi > 0$ there is a $\sigma$ and $\omega_0$ such that for each $\omega \geq \omega_0$ there is a strategy for Presenter to present intervals of length at most $\sigma$ and force Algorithm to use $\brac{\frac{1+\sqrt{5}}{2}-\epsi}\cdot\omega$ colors on an $\omega$-colorable set of intervals.

\begin{observation}\label{obs:phi}
There is no online algorithm that works for all $\sigma \geq 1$ and uses at most
$1.618\cdot\omega + f(\sigma)$ colors for $\omega$-colorable graphs
(for any function $f$).
\end{observation}

\subsubsection*{Overview of the \texorpdfstring{$5/2$}{5/2} Lower Bound.}

Our $5/2$ lower bound also works by recursively combining strategies, however the three phases are substantially different than above. In particular, in the first phase, instead of playing the simpler strategy only once, Presenter plays many instances of it, independently, side by side. Intuitively, either Algorithm uses mostly different colors for different instances, and it already uses too many colors, or there must be many colors shared by many instances.

So far the argument resembles the lower bound for interval graphs by Kierstead and Trotter~\cite{KiersteadT81}. They use pigeonhole principle to argue that if the total number of colors is bounded, then after playing exponentially many instances, there must be four instances that use exactly the same subset of colors. Thus, the number of instances -- and consequently the length of intervals in the subsequent phases -- required in their approach grows with the number of colors, which makes it infeasible in our setting.

To overcome the above difficulty we show (see Lemmas~\ref{lem:4sets} and~\ref{lem:4split}) that, if the total number of colors is small, there must be four groups of consecutive instances, which do not necessarily use exactly the same colors, but at least have large intersection.

\subsection{Preliminaries}

To properly capture asymptotic properties of the introduced strategies we give the following formal definitions.
\begin{definition}
\label{def:strategy}
For $\omega, C \in \PosNat$ and $\sigma, M \in \PosReal$, an \emph{$\tbrac{\omega,C,\sigma,M}$-strategy} is a strategy for Presenter that forces Algorithm to use at least $C$ colors subject to the following constraints:
\begin{enumerate}
\item the set of introduced intervals is $\omega$-colorable,
\item every introduced interval has length at least $1$ and at most $\sigma$,
\item every introduced interval is contained in the interval $\sbrac{0,M}$.
\end{enumerate}
\end{definition}
We are interested in providing strategies that achieve the biggest possible
ratio $\frac{C}{\omega}$ for large $\omega$. This motivates the following definition.
\begin{definition}
\label{def:schema}
An \emph{$\tbrac{\alpha,\sigma,M}$-schema} is a set of $\tbrac{\omega,C_\omega,\sigma,M}$-strategies for all $\omega \in \PosNat$ such that $C_\omega = \alpha\omega - \oh{\omega}$.
\end{definition}

The $\oh{\omega}$ term in the above definition accounts for the fact that sometimes in a proof we would like to introduce, say, $\frac{\omega}{2}$-clique. Then, for odd $\omega$'s a rounding is required, which results in small inaccuracies we need to control.

\begin{remark}
Note that the existence of an \emph{$\tbrac{\alpha,\sigma,M}$-schema} implies
a lower bound of $\alpha$ for the asymptotic competitive ratio of any online
algorithm solving the $\sigma$-interval coloring problem.
\end{remark}

To put the above definitions in context, note that Kierstead and Trotter~\cite{KiersteadT81} give, for all $\omega \in \PosNat$, an $\tbrac{\omega,3\omega-2,f(\omega),f(\omega)}$-strategy. 
However, their family of strategies does not yield an $\tbrac{\alpha, \sigma,M}$-schema, because the length of the presented intervals grows with $\omega$.

\begin{example}[$\tbrac{1,1,1}$-schema]
For any $\omega \in \PosNat$, a strategy that introduces the interval $\sbrac{0,1}$ in every round $1,\ldots,\omega$ is an $\tbrac{\omega,\omega,1,1}$-strategy.
The set of these strategies is a $\tbrac{1,1,1}$-schema.
\end{example}

In this section we show a series of constructions that use an existing schema to create another schema with different parameters.
The $\tbrac{1,1,1}$-schema given above is the initial step for those constructions.

Let $S$ be an $\tbrac{\omega,C,\sigma,M}$-strategy.
We will say that \emph{Presenter uses strategy $S$ in the interval $\sbrac{x,x+M}$} to denote that Presenter plays according to $S$, presenting intervals shifted by $x$, until Algorithm uses $C$ colors.

\subsection{Warm-up}\label{sec:lower_warmup}

Our first construction is a natural generalization of the strategy for unit 
intervals given by Epstein and Levy~\cite{EpsteinL05}, already described in Section~\ref{sec:reccomp} in an informal way.
The construction is surpassed by more involved strategies coming later, but it serves as a gentle introduction to our formal framework.


\begin{lemma}\label{lem:lower_32}
If there is an $\tbrac{\alpha,\sigma,M}$-schema, then there is a $\tbrac{2-\frac{1}{\alpha+1}, M+\epsi, M+1+\epsi}$-schema for every $\epsi > 0$.
\end{lemma}

\begin{proof}
To prove the lemma we need to provide an $\tbrac{\omega,(2-\frac{1}{\alpha+1})\omega-\oh{\omega}, M+\epsi, M+1+\epsi}$-strategy for every $\omega \in \PosNat$.
Let us fix an $\omega \in \PosNat$, and let $\omega' = \floor{\frac{\omega}{\alpha+1}}$.
The $\tbrac{\alpha,\sigma,M}$-schema contains an $\tbrac{\omega',\alpha\omega' - \delta, \sigma, M}$-strategy $S$ for some $\delta = \oh{\omega'}$.
The strategy for Presenter consists of three phases (see Figure~\ref{fig:lower_32}).
In the first phase, called the \emph{initial phase}, Presenter uses strategy $S$ inside the interval $\sbrac{1+\epsi, M+1+\epsi}$.
Let $C = \alpha \omega' - \delta$ and let $\mathcal{X}$ denote the set of $C$ colors used by Algorithm in the initial phase.

\begin{figure}
\centering
\begin{tikzpicture}
\tiny
\draw [<->] (-1em,0em) -- (-1em,-15em); \node at (-1em,-7.5em) [anchor=east] {$\omega$};
\node (split1) at (0em,0em) [
    anchor=north west,
    rectangle,
    draw,
    dashed,
    align=left,
    minimum width=12em,
    minimum height=8em,
]{
    \emph{separation phase}\\
    width: $\omega-\omega'$
};
\draw [<->] (13em, 1em) -- (29em,1em); \node at (21em, 1em) [anchor=south] {$M+\epsi$};
\node (final) at (13em,0em)[
    anchor=north west,
    rectangle,
    draw,
    dashed,
    align=left,
    minimum width=16em,
    minimum height=8em,
]{
    \emph{final phase}\\
    width: $\omega-\omega'$\\
    colors: $\norm{\mathcal{Z}} = \omega-\omega'$
};
\node (split2) at (3em,-9em)[
    anchor=north west,
    rectangle,
    draw,
    dashed,
    align=left,
    minimum width=12em,
    minimum height=6em,
]{
    \emph{separation phase}\\
    width: $\omega'$\\
    colors: $\norm{\mathcal{Y}}=\omega'$
};
\draw [<->] (30em, -9em) -- (30em,-15em); \node at (30em, -12em) [anchor=west] {$\omega'$};
\node (init) at (16em, -9em)[
    anchor=north west,
    rectangle,
    draw,
    dashed,
    align=left,
    minimum width=13em,
    minimum height=6em,
]{
    \emph{initial phase}\\
    width: $\omega'$\\
    colors: $\norm{\mathcal{X}}=\alpha\omega'-\delta$
};
\draw[<->] (16em, -16em) -- (29em,-16em); \node at (22.5em,-16em) [anchor=north] {$M$};
\draw[<->] (0em, -16em) -- (15em,-16em); \node at (7.5em,-16em) [anchor=north] {$1+\frac{\epsi}{2}$};
\draw[<->] (0em, -18em) -- (29em,-18em); \node at (14.5em,-18em) [anchor=north] {$M+1+\epsi$};
\end{tikzpicture}
\caption{Strategy construction in Lemma~\ref{lem:lower_32}}
\label{fig:lower_32}
\end{figure}

The second phase, borrowed from~\cite{EpsteinL05,Azar06}, is called the \emph{separation phase}.
In this phase, Presenter plays the following separation strategy for $\omega$ rounds.
Let $l_1 = 0$ and $r_1 = \frac{\epsi}{2}$.
In the $i$-th round of the separation phase Presenter introduces the interval $[\frac{l_i+r_i}{2}, \frac{l_i+r_i}{2}+1]$.
If Algorithm colors the interval with one of the colors in $\mathcal{X}$, let $l_{i+1} = \frac{l_i+r_i}{2}$ and $r_{i+1} = r_i$, which means that the next interval will be shifted slightly to the right.
Otherwise, let $l_{i+1} = l_i$ and $r_{i+1} = \frac{l_i+r_i}{2}$, which means that the next interval will be shifted slightly to the left.

The above procedure guarantees the following invariant. At the beginning of round $i$ all the previously introduced intervals with a color in $\mathcal{X}$ have their left endpoints to the left of $l_i$, and, conversely, all the previously introduced intervals with a color not in $\mathcal{X}$ have their left endpoints to the right of $r_i$. Moreover, all the intervals yet to be introduced will have their left endpoints strictly between $l_i$ and $r_i$.

Observe that all intervals introduced in the separation phase have length $1$ and $\forall_i \frac{l_i+r_i}{2} < \frac{\epsi}{2}$.
Thus, every interval introduced in the separation phase is contained in $\left[0,1+\frac{\epsi}{2}\right]$ and any two of those intervals intersect.
Furthermore, the above invariant guarantees that for any two intervals $x$, $y$ introduced in the separation phase, $x$ colored with a color in $\mathcal{X}$, and $y$ colored with a color not in $\mathcal{X}$, we have that the left endpoint of $x$ is to the left of the left endpoint of $y$.
Let $Y$ be the set of $\omega' \leq \omega - |\mathcal{X}|$ right-most intervals introduced in the separation phase, and let $\mathcal{Y}$ be the set of colors used by Algorithm on the intervals in $Y$. Note that $\mathcal{X}$ and $\mathcal{Y}$ are disjoint.

For the last phase, called the \emph{final phase}, let $r$ be the left-most right endpoint of an interval in $Y$.
In the final phase Presenter introduces $\omega - \omega'$ times the same interval $\sbrac{r, M+1+\epsi}$.
This interval intersects all intervals introduced in the initial phase, all intervals in $Y$, and no other interval introduced in the separation phase.
Thus, Algorithm must use $\omega - \omega'$ colors in the final phase that are different from the colors in both $\mathcal{X}$ and $\mathcal{Y}$.
Let $\mathcal{Z}$ denote the set of colors used by Algorithm in the final phase.

The presented set of intervals is clearly $\omega$-colorable and Algorithm used at least $\norm{\mathcal{X}}+\norm{\mathcal{Y}}+\norm{\mathcal{Z}} = \alpha\omega'-\delta+\omega'+\omega-\omega' = \brac{2-\frac{1}{\alpha+1}}\omega - \oh{\omega}$ many colors.
The longest interval presented has length $M+\epsi$, and all intervals are contained in $\sbrac{0,M+1+\epsi}$.
Thus, we have constructed a $\tbrac{2-\frac{1}{\alpha+1}, M+\epsi, M+1+\epsi}$-schema.
\end{proof}

\subsection{The \texorpdfstring{$5/3$}{5/3} Lower Bound}\label{sec:lower_53}


\begin{lemma}\label{lem:lower_53}
If there is an $\tbrac{\alpha,\sigma,M}$-schema, then there is a $\tbrac{2-\frac{1}{\alpha+2}, M+\epsi, M+2+\epsi}$-schema for every $\epsi > 0$.
\end{lemma}

\begin{proof}
The proof of this lemma is very similar to the proof of Lemma~\ref{lem:lower_32}, but now we have two separation phases instead of just one, see Figure~\ref{fig:lower_53}.
Let us fix an $\omega \in \PosNat$, and let $\omega' = \floor{\frac{\omega}{\alpha+2}}$.
Let $S$ be an $\tbrac{ \omega',\alpha\omega' - \delta, \sigma, M}$-strategy for some $\delta = \oh{\omega'}$.

\begin{figure}
\centering
\begin{tikzpicture}
\tiny
\draw [<->] (-1em,0em) -- (-1em,-15em); \node at (-1em,-7.5em) [anchor=east] {$\omega$};
\node (splita1) at (0em,0em) [
    anchor=north west,
    rectangle,
    draw,
    dashed,
    align=left,
    minimum width=10em,
    minimum height=8em,
]{
    \emph{separation phase}\\
    width: $\omega-\omega'$
};
\draw [<->] (11em, 1em) -- (25em,1em); \node at (18em, 1em) [anchor=south] {$M+\epsi'$};
\node (final) at (11em,0em)[
    anchor=north west,
    rectangle,
    draw,
    dashed,
    align=left,
    minimum width=14em,
    minimum height=8em,
]{
    \emph{final phase}\\
    width: $\omega-\omega'$\\
    colors: $\norm{\mathcal{Z}} = \omega-\omega'$
};
\node (splitb1) at (26em,0em) [
    anchor=north west,
    rectangle,
    draw,
    dashed,
    align=left,
    minimum width=10em,
    minimum height=8em,
]{
    \emph{separation phase}\\
    width: $\omega-\omega'$
};
\node (splita2) at (2em,-9em)[
    anchor=north west,
    rectangle,
    draw,
    dashed,
    align=left,
    minimum width=10em,
    minimum height=6em,
]{
    \emph{separation phase}\\
    width: $\omega'$\\
    colors: $\norm{\mathcal{Y}_1}=\omega'$
};
\node (init) at (13em, -9em)[
    anchor=north west,
    rectangle,
    draw,
    dashed,
    align=left,
    minimum width=10em,
    minimum height=6em,
]{
    \emph{initial phase}\\
    width: $\omega'$\\
    colors: $\norm{\mathcal{X}}=\alpha\omega'-\delta$
};
\draw [<->] (37em, -9em) -- (37em,-15em); \node at (37em, -12em) [anchor=west] {$\omega'$};
\node (splitb2) at (24em,-9em)[
    anchor=north west,
    rectangle,
    draw,
    dashed,
    align=left,
    minimum width=10em,
    minimum height=6em,
]{
    \emph{separation phase}\\
    width: $\omega'$\\
    colors: $\norm{\mathcal{Y}_2}=\omega'$
};

\draw[<->] (13em, -16em) -- (23em,-16em); \node at (18em,-16em) [anchor=north] {$M$};
\draw[<->] (0em, -16em) -- (12em,-16em); \node at (6em,-16em) [anchor=north] {$1+\frac{\epsi}{4}$};
\draw[<->] (24em, -16em) -- (36em,-16em); \node at (30em,-16em) [anchor=north] {$1+\frac{\epsi}{4}$};
\draw[<->] (0em, -18em) -- (36em,-18em); \node at (18em,-18em) [anchor=north] {$M+2+\epsi$};
\end{tikzpicture}
\caption{Strategy construction in Lemma~\ref{lem:lower_53}}
\label{fig:lower_53}
\end{figure}

In the initial phase, Presenter uses $S$ inside interval
$\sbrac{1+\frac{\epsi}{2}, M+1+\frac{\epsi}{2}}$, and forces Algorithm to use
$C=\alpha\omega'-\delta$ colors. Let $\mathcal{X}$ denote the set of those colors.

In the separation phase, Presenter plays the separation strategy two times.
First, Presenter plays the separation strategy for $\omega$ rounds in the region $\sbrac{0,1+\frac{\epsi}{4}}$ pushing to the right colors not in $\mathcal{X}$.
Let $Y_1$ be the set of $\omega'$ right-most intervals from this first separation.
Let $\mathcal{Y}_1$ denote the set of colors used by Algorithm to color $Y_1$.
Then, Presenter plays the separation strategy for $\omega$ rounds in the region $\sbrac{M+1+\frac{3\epsi}{4},M+2+\epsi}$ pushing to the left colors not in $\mathcal{X}\cup\mathcal{Y}_1$.
Let $Y_2$ be the set of $\omega'$ left-most intervals from this second separation.
Let $\mathcal{Y}_2$ denote the set of colors used by Algorithm to color $Y_2$.

Let $r$ be the left-most right endpoint of an interval in $Y_1$.
Let $l$ be the right-most left endpoint of an interval in $Y_2$.
In the final phase Presenter introduces $\omega - \omega'$ times the same interval $\sbrac{r, l}$.

The presented set of intervals is clearly $\omega$-colorable and Algorithm used at least $\norm{\mathcal{X}}+\norm{\mathcal{Y}_1}+\norm{\mathcal{Y}_2}+\norm{\mathcal{Z}} = \alpha\omega'-\delta+\omega'+\omega'+\omega-\omega' = \brac{2-\frac{1}{\alpha+2}}\omega - \oh{\omega}$ many colors.
The longest interval presented has length at most $M+\epsi$, and all intervals are contained in $\sbrac{0,M+2+\epsi}$.
Thus, we have constructed a $\tbrac{2-\frac{1}{\alpha+2}, M+\epsi, M+2+\epsi}$-schema.
\end{proof}

\begin{corollary}\label{cor:lower_53}
There is an $\tbrac{\alpha_{n}, 2n-1+\epsi, 2n+1+\epsi}$-schema,
for every $n \in \PosNat$ and every $\epsi > 0$, where
\[
\alpha_{n} = \frac
  {\brac{\sqrt{3}-3}\brac{\sqrt{3}-2}^{n}+\brac{\sqrt{3}+3}\brac{-\sqrt{3}-2}^{n}} 
  {\brac{\sqrt{3}-1}\brac{\sqrt{3}-2}^{n}+\brac{\sqrt{3}+1}\brac{-\sqrt{3}-2}^{n}}
\text{.}\]
\end{corollary}

\begin{proof}

Starting with a $\tbrac{1,1,1}$-schema and repeatedly applying Lemma~\ref{lem:lower_53} one can generate\footnote{Knowing the desired target values of $n$ and $\epsilon$, one needs to properly adjust the $\epsilon$ value for each application of Lemma~\ref{lem:lower_53}, e.g., it is sufficient to set it to $\epsilon/n$.}
a family of $\tbrac{\alpha_n,\sigma_n+\epsi_n,M_n+\epsi_n}$-schemas, such that
$\alpha_{n+1} = 2 - \frac{1}{\alpha_{n}+2}$,
$\sigma_{n+1} = M_{n}$,
$M_{n+1} = M_{n} + 2$,
and $\alpha_0 = \sigma_0 = M_0 = 1$.
Solving the recurrence equations we get
$\alpha_{n} = \frac{F_{2n+1}}{F_{2n}}$,
$\sigma_{n} = 2n-1$,
and $M_{n} = 2n+1$.
\end{proof}

Note that, similarly to Observation~\ref{obs:phi}, one could already use
Corollary~\ref{cor:lower_53} to get a lower bound arbitrarily close to
$\lim_{n\to\infty}\alpha_n=\sqrt{3} \approx 1.73205$ for the asymptotic competitive ratio of any
online algorithm that work for all $\sigma \ge 1$.
Nonetheless in Section~\ref{sec:lower_52} we prove a stronger $5/2$ lower bound.

\FiveThirdsThm*

\begin{proof}
Assume for contradiction that for some $\sigma > 1$ there exists an online algorithm for $\sigma$-interval coloring with the asymptotic competitive ratio $\frac{5}{3}-\epsi$, for some $\epsi > 0$. By the definition of the asymptotic competitive ratio, there is an $\omega_A$ such that for every $\omega \geq \omega_A$ the algorithm colors every $\omega$-colorable set of intervals using at most $\brac{\frac{5}{3}-\epsi+\frac{\epsi}{3}}\omega = \brac{\frac{5}{3}-\frac{2\epsi}{3}}\omega$ colors.

Observe that, for $n=1$, Corollary~\ref{cor:lower_53} gives a $\tbrac{\frac{5}{3},1+(\sigma-1),3+(\sigma-1)}$-schema.
By the definition of schema, there is an $\omega_P$ such that for every $\omega \geq \omega_P$ there is a strategy for Presenter to present an $\omega$-colorable set of intervals, of length in $\sbrac{1,\sigma}$, and force Algorithm to use $\brac{\frac{5}{3}-\frac{\epsi}{3}}\omega$ colors. For $\omega=\max(\omega_A,\omega_P)$ we reach a contradiction.
\end{proof}

\subsection{The \texorpdfstring{$7/4$}{7/4} Lower Bound}\label{sec:lower_74}


\begin{lemma}\label{lem:lower_74}
If there is an $\tbrac{\alpha,\sigma,M}$-schema, then there is a $\tbrac{\frac{2\alpha+1}{2\alpha+2}, 2M+\epsi, 2M+2+\epsi}$-schema for every $\epsi > 0$.
\end{lemma}

\begin{proof}
The proof of this lemma is a bit more complicated than the previous ones, as we now have two initial phases, two separation phases and a strategy branching, see Figure~\ref{fig:lower_74_case1} and Figure~\ref{fig:lower_74_case2}.
Let us fix an $\omega \in \PosNat$, and let $\omega' = \floor{\frac{\omega}{\alpha+1}}$.
Let $S$ be an $\tbrac{\omega',\alpha\omega' - \delta, \sigma, M}$-strategy for some $\delta = \oh{\omega'}$.

In the initial phase, Presenter uses strategy $S$ twice: (1) inside interval $\sbrac{1+\frac{\epsi}{3},M+1+\frac{\epsi}{3}}$, and (2) inside interval $\sbrac{M+1+\frac{2\epsi}{3}, 2M+1+\frac{2\epsi}{3}}$.
Algorithm uses $C=\alpha\omega'-\delta$ colors in each of these intervals.
We get a set of colors $\mathcal{X}_1$ used by Algorithm in the first interval, and a set of colors $\mathcal{X}_2$ used by Algorithm in the second interval.
Note that $\mathcal{X}_1\cap\mathcal{X}_2$ might be non-empty.

In the separation phase, Presenter plays the separation strategy two times.
First, Presenter plays the separation strategy for $\omega$ rounds in the region $\sbrac{0,1+\frac{\epsi}{6}}$ pushing to the right colors not in $\mathcal{X}_1$.
Let $Y_1$ be the set of $\omega'$ right-most intervals from the first separation phase.
Let $\mathcal{Y}_1$ denote the set of colors used by Algorithm to color $Y_1$.
Then, Presenter plays the separation strategy for $\omega$ rounds in the region $\sbrac{2M+1+\frac{5\epsi}{6},2M+2+\epsi}$ pushing to the left colors not in $\mathcal{X}_2$.
Let $Y_2$ be the set of $\omega'$ left-most intervals from the second separation phase.
Let $\mathcal{Y}_2$ denote the set of colors used by Algorithm to color $Y_2$.
Let $r$ be the left-most right endpoint of an interval in $Y_1$.
Let $l$ be the right-most left endpoint of an interval in $Y_2$.

There are two cases in the final phase.
Let $\mathcal{C}_{1} := \mathcal{X}_{1} \cup \mathcal{Y}_{1}$, and analogously $\mathcal{C}_{2} := \mathcal{X}_{2} \cup \mathcal{Y}_{2}$.
    We have that $\norm{\mathcal{C}_1} = \norm{\mathcal{C}_2} = \brac{\alpha+1}\omega' - \delta = \omega - \oh{\omega}$.

\begin{figure}
\centering
\begin{tikzpicture}
\tiny
\draw [<->] (-1em,0em) -- (-1em,-15em); \node at (-1em,-7.5em) [anchor=east] {$\omega$};
\node (splita1) at (0em,0em) [
    anchor=north west,
    rectangle,
    draw,
    dashed,
    align=left,
    minimum width=10em,
    minimum height=8em,
]{
    \emph{separation phase}\\
    width: $\omega-\omega'$
};
\draw [<->] (11em, 1em) -- (38em,1em); \node at (24.5em, 1em) [anchor=south] {$2M+\epsi'$};
\node (final) at (11em,0em)[
    anchor=north west,
    rectangle,
    draw,
    dashed,
    align=left,
    minimum width=27em,
    minimum height=8em,
]{
    \emph{final phase}\\
    width: $\omega-\omega'$\\
    colors: $\norm{\mathcal{Z}} = \omega-\omega'$
};
\node (splitb1) at (39em,0em) [
    anchor=north west,
    rectangle,
    draw,
    dashed,
    align=left,
    minimum width=10em,
    minimum height=8em,
]{
    \emph{separation phase}\\
    width: $\omega-\omega'$
};
\node (splita2) at (2em,-9em)[
    anchor=north west,
    rectangle,
    draw,
    dashed,
    align=left,
    minimum width=10em,
    minimum height=6em,
]{
    \emph{separation phase}\\
    width: $\omega'$\\
    colors: $\norm{\mathcal{Y}_1}=\omega'$
};
\node (init1) at (13em, -9em)[
    anchor=north west,
    rectangle,
    draw,
    dashed,
    align=left,
    minimum width=10em,
    minimum height=6em,
]{
    \emph{initial phase}\\
    width: $\omega'$\\
    colors:\\$\norm{\mathcal{X}_1}=\alpha\omega'-\delta$
};
\node (init2) at (26em, -9em)[
    anchor=north west,
    rectangle,
    draw,
    dashed,
    align=left,
    minimum width=10em,
    minimum height=6em,
]{
    \emph{initial phase}\\
    width: $\omega'$\\
    colors:\\$\norm{\mathcal{X}_2}=\alpha\omega'-\delta$
};
\draw [<->] (50em, -9em) -- (50em,-15em); \node at (50em, -12em) [anchor=west] {$\omega'$};

\node (splitb2) at (37em,-9em)[
    anchor=north west,
    rectangle,
    draw,
    dashed,
    align=left,
    minimum width=10em,
    minimum height=6em,
]{
    \emph{separation phase}\\
    width: $\omega'$\\
    colors: $\norm{\mathcal{Y}_2}=\omega'$
};

\draw[<->] (13em, -16em) -- (23em,-16em); \node at (18em,-16em) [anchor=north] {$M$};
\draw[<->] (26em, -16em) -- (36em,-16em); \node at (31em,-16em) [anchor=north] {$M$};
\draw[<->] (0em, -16em) -- (12em,-16em); \node at (6em,-16em) [anchor=north] {$1+\frac{\epsi}{6}$};
\draw[<->] (37em, -16em) -- (49em,-16em); \node at (43em,-16em) [anchor=north] {$1+\frac{\epsi}{6}$};
\draw[<->] (0em, -18em) -- (49em,-18em); \node at (24.5em,-18em) [anchor=north] {$2M+2+\epsi$};
\end{tikzpicture}
\caption{Lemma~\ref{lem:lower_74}, Case~1: $\norm{\mathcal{C}_{2} \setminus \mathcal{C}_{1}} \geq \frac{\omega}{2\alpha+2}$}
\label{fig:lower_74_case1}
\end{figure}

\subsubsection*{Case 1.}
If $\norm{\mathcal{C}_{2} \setminus \mathcal{C}_{1}} \geq \frac{\omega}{2\alpha+2}$,
then Presenter introduces $\omega - \omega'$ times the same interval $\sbrac{r, l}$.

Each interval introduced in the final phase intersects with all intervals from both initial phases and all intervals in $Y_1 \cup Y_2$.
Thus, Algorithm is forced to use $\norm{\mathcal{C}_{1} \cup \mathcal{C}_{2}} + \omega - \omega' = \norm{\mathcal{C}_{1}} + \norm{\mathcal{C}_{2} \setminus \mathcal{C}_{1}} + \omega-\omega' \geq \omega -\oh{\omega} + \frac{\alpha+\frac{1}{2}}{\alpha+1}\omega = \brac{2-\frac{1}{2\alpha+2}}\omega -\oh{\omega}$ colors in total.

\begin{figure}
\centering
\begin{tikzpicture}
\tiny
\draw [<->] (-1em,0em) -- (-1em,-15em); \node at (-1em,-7.5em) [anchor=east] {$\omega$};
\node (splita1) at (0em,0em) [
    anchor=north west,
    rectangle,
    draw,
    dashed,
    align=left,
    minimum width=10em,
    minimum height=8em,
]{
    \emph{separation phase}\\
    width: $\omega-\omega'$
};
\draw [<->] (11em, 1em) -- (38em,1em); \node at (24.5em, 1em) [anchor=south] {$2M+\epsi'$};
\node (prefinal) at (24em,-4.5em)[
    anchor=north west,
    rectangle,
    draw,
    dashed,
    align=left,
    minimum width=14em,
    minimum height=3.5em,
]{
    \emph{pre-final phase}\\
    width: $\omega'$
};
\node (final) at (11em,-0em)[
    anchor=north west,
    rectangle,
    draw,
    dashed,
    align=left,
    minimum width=14em,
    minimum height=3.5em,
]{
    \emph{final phase}\\
    width: $\omega-\omega'$
};

\node (splitb1) at (39em,0em) [
    anchor=north west,
    rectangle,
    draw,
    dashed,
    align=left,
    minimum width=10em,
    minimum height=8em,
]{
    \emph{separation phase}\\
    width: $\omega-\omega'$
};
\node (splita2) at (2em,-9em)[
    anchor=north west,
    rectangle,
    draw,
    dashed,
    align=left,
    minimum width=10em,
    minimum height=6em,
]{
    \emph{separation phase}\\
    width: $\omega'$\\
    colors: $\norm{\mathcal{Y}_1}=\omega'$
};
\node (init1) at (13em, -9em)[
    anchor=north west,
    rectangle,
    draw,
    dashed,
    align=left,
    minimum width=10em,
    minimum height=6em,
]{
    \emph{initial phase}\\
    width: $\omega'$\\
    colors:\\$\norm{\mathcal{X}_1}=\alpha\omega'-\delta$
};
\node (init2) at (26em, -9em)[
    anchor=north west,
    rectangle,
    draw,
    dashed,
    align=left,
    minimum width=10em,
    minimum height=6em,
]{
    \emph{initial phase}\\
    width: $\omega'$\\
    colors:\\$\norm{\mathcal{X}_2}=\alpha\omega'-\delta$
};
\draw [<->] (50em, -9em) -- (50em,-15em); \node at (50em, -12em) [anchor=west] {$\omega'$};

\node (splitb2) at (37em,-9em)[
    anchor=north west,
    rectangle,
    draw,
    dashed,
    align=left,
    minimum width=10em,
    minimum height=6em,
]{
    \emph{separation phase}\\
    width: $\omega'$\\
    colors: $\norm{\mathcal{Y}_2}=\omega'$
};

\draw[<->] (13em, -16em) -- (23em,-16em); \node at (18em,-16em) [anchor=north] {$M$};
\draw[<->] (26em, -16em) -- (36em,-16em); \node at (31em,-16em) [anchor=north] {$M$};
\draw[<->] (0em, -16em) -- (12em,-16em); \node at (6em,-16em) [anchor=north] {$1+\frac{\epsi}{6}$};
\draw[<->] (37em, -16em) -- (49em,-16em); \node at (43em,-16em) [anchor=north] {$1+\frac{\epsi}{6}$};
\draw[<->] (0em, -18em) -- (49em,-18em); \node at (24.5em,-18em) [anchor=north] {$2M+2+\epsi$};
\end{tikzpicture}
\caption{Lemma~\ref{lem:lower_74}, Case~2: $\norm{\mathcal{C}_{2} \setminus \mathcal{C}_{1}} < \frac{\omega}{2\alpha+2}$}
\label{fig:lower_74_case2}
\end{figure}

\subsubsection*{Case 2.}
If $\norm{\mathcal{C}_{2} \setminus \mathcal{C}_{1}} < \frac{\omega}{2\alpha+2}$,
then Presenter introduces $\omega'$ intervals, all of them having endpoints $\sbrac{M+1+5\epsi/12,l}$.
Let $Q$ be the set of colors used by Algorithm in this \emph{pre-final phase}.
We have $\mathcal{C}_{2} \cap Q = \emptyset$,
and we assumed that $\norm{\mathcal{C}_{2} \setminus \mathcal{C}_{1}} \leq \frac{\omega}{2\alpha+2}$,
thus we have $\norm{Q \setminus \mathcal{C}_{1}} \geq \frac{\omega}{2\alpha+2}$,
and now we are in a situation analogous to Case 1, with $Q$ playing the role of $\mathcal{C}_{2}$, see Figure~\ref{fig:lower_74_case2}.

The longest interval introduced by Presenter in both cases has length strictly less than $2M+\epsi$, and the whole game is played in the region $\sbrac{0,2M+2+\epsi}$.
\end{proof}

\begin{corollary}\label{cor:lower_74}
There is an $\tbrac{\alpha_{n}, 3\cdot2^{n}-4+\epsi, 3\cdot2^{n}-2+\epsi}$-schema,
for every $n \in \PosNat$ and every $\epsi > 0$, where
\[
    \alpha_{n} = \frac
  {\brac{\sqrt{7}-4}\brac{\sqrt{7}-3}^{n}+\brac{\sqrt{7}+4}\brac{-\sqrt{7}-3}^{n}} 
  {\brac{\sqrt{7}-1}\brac{\sqrt{7}-3}^{n}+\brac{\sqrt{7}+1}\brac{-\sqrt{7}-3}^{n}}
\text{.}\]
\end{corollary}

\begin{proof}
The argument is similar to Corollary~\ref{cor:lower_53}, but now we solve the recurrence equations
$\alpha_{0} = 1$, $\alpha_{n+1} = 2 - \frac{1}{2\alpha_{n}+2}$,
and $M_{0} = 1$, $M_{n+1} = 2M_{n} + 2$, $\sigma_{0} = 1, \sigma_{n+1} = 2M_n$.
\end{proof}

Note that, similarly to Observation~\ref{obs:phi}, one could already use
Corollary~\ref{cor:lower_74} to get a lower bound arbitrarily close to
$\lim_{n\to\infty}\alpha_n=\frac{1+\sqrt{7}}{2} \approx 1.82287$ for the asymptotic competitive
ratio of any online algorithm that work for all $\sigma \ge 1$.
Nonetheless in Section~\ref{sec:lower_52} we prove a stronger $5/2$ lower bound.

\SevenFourthsThm*

\begin{proof}
Observe that, for $\sigma > 2$ and $n=1$, Corollary~\ref{cor:lower_74} gives a $\tbrac{\frac{7}{4},2+(\sigma-2),4+(\sigma-2)}$-schema.
Then proceed analogously to the proof of Theorem~\ref{thm:53}.
\end{proof}

\subsection{The \texorpdfstring{$5/2$}{5/2} Lower Bound}\label{sec:lower_52}

To prove our main negative result we need two simple combinatorial lemmas.


\begin{lemma}\label{lem:4sets}
Let $\gamma \in \sbrac{0,1}$. For every four sets $X_1,\ldots,X_4$,
each of size $k$, if their intersection is small:
$\norm{\bigcap_{i=1}^{4}X_i} \leq \brac{1-\gamma}\cdot k$,
their union is large: $\norm{\bigcup_{i=1}^{4}X_i} \geq \frac{3+\gamma}{3} \cdot k$.
\end{lemma}

\begin{proof}
Each element which belongs to the union but does not belong to the intersection can belong to at most three sets.
Thus, we have 
\[
    3\cdot\brac{\norm{\bigcup_{i=1}^{4}X_i} - \norm{\bigcap_{i=1}^{4}X_i}}
    \geq
    4\cdot\brac{k - \norm{\bigcap_{i=1}^{4}X_i}}
    \text{,}
\]
and so 
\[
    3\cdot\norm{\bigcup_{i=1}^{4}X_i}
    \geq
    4k-\norm{\bigcap_{i=1}^{4}X_i}
    \geq
    \brac{3+\gamma}\cdot k
    \text{.}
\]
\end{proof}


\begin{lemma}\label{lem:4split}
Let $\gamma \in \sbrac{0,1}$, and $X_1,\ldots,X_{4^n}$ be a family of $4^n$ sets, each of size $k$. Then, either 
\[\norm{\bigcup_{i=1}^{4^n}X_i} \geq \brac{\frac{3+\gamma}{3}}^{n}k\text{,}\]
or the sequence $1,2,\ldots,4^n$ can be covered with four disjoint intervals
$[l_1,r_1], \ldots, [l_4,r_4]$, $l_1=1$, $l_{i+1}=r_i+1$, $r_4=4^n$, such that
for $Y_i = \bigcup_{j=l_i}^{r_i}X_j$ the intersection of $Y_i$'s is large:
\[\norm{Y_1 \cap Y_2 \cap Y_3 \cap Y_4} \geq (1-\gamma)\cdot k\text{.}\]
\end{lemma}

\begin{proof}
Consider $n+1$ families of sets defined as follows:
$\mathcal{X}_{i}^{0} := X_i$ for every $i \in[4^n]$, and
$\mathcal{X}_{i}^{j} := \bigcup_{l=4i-3}^{4i}\mathcal{X}_{l}^{j-1}$
for every $j \in [n]$ and $i \in [4^{n-j}]$. See Figure~\ref{fig:4split}.

\begin{figure}
\centering
\begin{tikzpicture}
\tiny
    \node at (-0.5em,5.3em) [
        anchor=north west,rectangle,draw,dashed,align=left,
        minimum width=43em,
        minimum height=8.2em
    ]{};
    \node at (-0.2em,5em) [
        anchor=north west,rectangle,align=left,
        minimum width=42.4em,
        minimum height=2.4em
    ]{$\mathcal{X}_1^2$};
    \begin{scope}[shift={(0em,0em)}]
    \node at (-0.2em,2.7em) [
        anchor=north west,rectangle,draw,dashed,align=left,
        minimum width=10.3em,
        minimum height=5.3em
    ]{};
    \node (x1) at (0em,0em) [
        anchor=north west,rectangle,draw,align=left,
        minimum width=2.4em,
        minimum height=2.4em
    ]{$X_{1}$};
    \node (x2) at (2.5em,0em) [
        anchor=north west,rectangle,draw,align=left,
        minimum width=2.4em,
        minimum height=2.4em
    ]{$X_{2}$};
    \node (x3) at (5em,0em) [
        anchor=north west,rectangle,draw,align=left,
        minimum width=2.4em,
        minimum height=2.4em
    ]{$X_{3}$};
    \node (x4) at (7.5em,0em) [
        anchor=north west,rectangle,draw,align=left,
        minimum width=2.4em,
        minimum height=2.4em
    ]{$X_{4}$};
    \node at (0em,2.5em) [
        anchor=north west,rectangle,align=left,
        minimum width=9.9em,
        minimum height=2.4em
    ]{$\mathcal{X}_1^1$};
    \end{scope}
    \begin{scope}[shift={(10.7em,0em)}]
    \node at (-0.2em,2.7em) [
        anchor=north west,rectangle,draw,dashed,align=left,
        minimum width=10.3em,
        minimum height=5.3em
    ]{};
    \node (x1) at (0em,0em) [
        anchor=north west,rectangle,draw,align=left,
        minimum width=2.4em,
        minimum height=2.4em
    ]{$X_{5}$};
    \node (x2) at (2.5em,0em) [
        anchor=north west,rectangle,draw,align=left,
        minimum width=2.4em,
        minimum height=2.4em
    ]{$X_{6}$};
    \node (x3) at (5em,0em) [
        anchor=north west,rectangle,draw,align=left,
        minimum width=2.4em,
        minimum height=2.4em
    ]{$X_{7}$};
    \node (x4) at (7.5em,0em) [
        anchor=north west,rectangle,draw,align=left,
        minimum width=2.4em,
        minimum height=2.4em
    ]{$X_{8}$};
    \node at (0em,2.5em) [
        anchor=north west,rectangle,align=left,
        minimum width=9.9em,
        minimum height=2.4em
    ]{$\mathcal{X}_2^1$};
    \end{scope}
    \begin{scope}[shift={(21.4em,0em)}]
    \node at (-0.2em,2.7em) [
        anchor=north west,rectangle,draw,dashed,align=left,
        minimum width=10.3em,
        minimum height=5.3em
    ]{};
    \node (x1) at (0em,0em) [
        anchor=north west,rectangle,draw,align=left,
        minimum width=2.4em,
        minimum height=2.4em
    ]{$X_{9}$};
    \node (x2) at (2.5em,0em) [
        anchor=north west,rectangle,draw,align=left,
        minimum width=2.4em,
        minimum height=2.4em
    ]{$X_{10}$};
    \node (x3) at (5em,0em) [
        anchor=north west,rectangle,draw,align=left,
        minimum width=2.4em,
        minimum height=2.4em
    ]{$X_{11}$};
    \node (x4) at (7.5em,0em) [
        anchor=north west,rectangle,draw,align=left,
        minimum width=2.4em,
        minimum height=2.4em
    ]{$X_{12}$};
    \node at (0em,2.5em) [
        anchor=north west,rectangle,align=left,
        minimum width=9.9em,
        minimum height=2.4em
    ]{$\mathcal{X}_3^1$};
    \end{scope}
    \begin{scope}[shift={(32.1em,0em)}]
    \node at (-0.2em,2.7em) [
        anchor=north west,rectangle,draw,dashed,align=left,
        minimum width=10.3em,
        minimum height=5.3em
    ]{};
    \node (x1) at (0em,0em) [
        anchor=north west,rectangle,draw,align=left,
        minimum width=2.4em,
        minimum height=2.4em
    ]{$X_{13}$};
    \node (x2) at (2.5em,0em) [
        anchor=north west,rectangle,draw,align=left,
        minimum width=2.4em,
        minimum height=2.4em
    ]{$X_{14}$};
    \node (x3) at (5em,0em) [
        anchor=north west,rectangle,draw,align=left,
        minimum width=2.4em,
        minimum height=2.4em
    ]{$X_{15}$};
    \node (x4) at (7.5em,0em) [
        anchor=north west,rectangle,draw,align=left,
        minimum width=2.4em,
        minimum height=2.4em
    ]{$X_{16}$};
    \node at (0em,2.5em) [
        anchor=north west,rectangle,align=left,
        minimum width=9.9em,
        minimum height=2.4em
    ]{$\mathcal{X}_4^1$};
    \end{scope}
\end{tikzpicture}
\caption{$\mathcal{X}_i^j$ sets in Lemma~\ref{lem:4split}}
\label{fig:4split}
\end{figure}
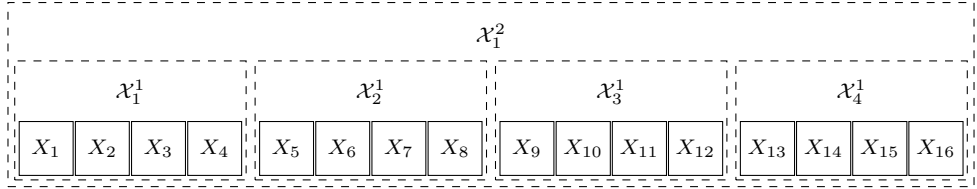

If for some $i,j$ we have $|\bigcap_{l=4i-3}^{4i} \mathcal{X}_{l}^{j}| \geq \brac{1-\gamma}\cdot k$, then we are done.
Thus, we assume that $\forall_{i,j}\ |\bigcap_{l=4i-3}^{4i} \mathcal{X}_{l}^{j}| < \brac{1-\gamma}\cdot k$.
Let $\rho := \frac{3+\gamma}{3} \in \sbrac{1,\frac{4}{3}}$.
We prove that $\forall_{i,j}\ |\mathcal{X}_{i}^{j}| \geq \rho^{j}k$, by induction on $j$.
For $j=0$ the statement is obvious because
$\forall_i\ |\mathcal{X}_i^0| = |X_i| = k = \rho^0 k$.
For $j+1$ and arbitrary $i$, let $k'=\rho^{j}k$. By the induction hypothesis $|\mathcal{X}_{4i-3}^{j}|,\ldots,|\mathcal{X}_{4i}^{j}| \geq \rho^{j}k = k'$.
We may ignore some elements of those sets and assume that $|\mathcal{X}_{4i-3}^{j}| = \ldots = |\mathcal{X}_{4i}^{j}| = k'$,
moreover we assumed that $|\mathcal{X}_{4i-3}^{j} \cap \ldots \cap \mathcal{X}_{4i}^{j}| < \brac{1-\gamma} k = \frac{1-\gamma}{\rho^j} \rho^{j}k = \brac{1-\gamma'}k'$, where $\gamma' \in \sbrac{0,1}$ and $\gamma' > \gamma$.
We apply Lemma~\ref{lem:4sets} and get $|\mathcal{X}_{4i-3}^{j} \cup \ldots \cup \mathcal{X}_{4i}^{j}| \geq \frac{3+\gamma'}{3}k'$.
Thus, $|\mathcal{X}_{i}^{j+1}| \geq \frac{3+\gamma'}{3}k' > \frac{3+\gamma}{3}k' = \rho k' = \rho^{j+1}k$.
\end{proof}


\begin{lemma}\label{lem:lower_52}
If there is an $\tbrac{\alpha,\sigma,M}$-schema, then for every $\epsi > 0$ and for every $\gamma \in \brac{0,1}$, there is a $\tbrac{\frac{5}{4}+\frac{1}{2}\brac{1-\gamma}\alpha, 4^{n}M+\epsi, 4^{n}M+\epsi}$-schema, for some $n := n\brac{\gamma}$.
\end{lemma}

\begin{proof}
Let us fix an $\omega \in \PosNat$, and let $\omega' = \floor{\frac{\omega}{2}}$.
Let $S$ be an $\tbrac{\omega',\alpha\omega' - \delta, \sigma, M}$-strategy for some $\delta = \oh{\omega'}$.
Presenter repeats strategy $S$ in the initial phase $4^n$ times. For each $i \in \sbrac{4^n}$ the $i$-th game is played inside interval $\sbrac{(i-1)(M+\frac{\epsi}{4^n}), (i-1)(M+\frac{\epsi}{4^n})+M}$. See Figure~\ref{fig:lower_52_case1}.
Algorithm uses $C=\alpha\omega'-\delta$ colors in each of these games.
Let $\mathcal{X}_i$ denote the set of $C$ colors used by Algorithm in the $i$-th game.
Let $\mathcal{X}$ denote the set of all colors used in the initial phase, i.e., $\mathcal{X} = \bigcup_{i \in \sbrac{4^n}} \mathcal{X}_i$.

We apply Lemma~\ref{lem:4split} to the family $\mathcal{X}_1,\ldots,\mathcal{X}_{4^n}$
and get that either
the union of these sets has at least $\brac{\frac{3+\gamma}{3}}^{n} C$ elements,
or we get four disjoint consecutive subfamilies $\mathcal{Y}_1, \ldots, \mathcal{Y}_4$
($\mathcal{Y}_i = \bigcup_{j=l_i}^{r_i}\mathcal{X}_j$)
such that the size of the intersection
$\mathcal{Y}_1 \cap \mathcal{Y}_2 \cap \mathcal{Y}_3 \cap \mathcal{Y}_4$
has at least $(1-\gamma) \cdot C$ elements.

\subsubsection*{Case 1.}
If the size of the union $\norm{\mathcal{X}}$ is at least $\brac{1+\frac{\gamma}{3}}^{n} \cdot C$,
then Presenter introduces $\omega'$ intervals, all of them having endpoints $\sbrac{0,4^{n}M+\epsi}$.
See Figure~\ref{fig:lower_52_case1}.
Each interval introduced in the final phase intersects with all intervals introduced in the initial phase.
Thus, Algorithm is forced to use at least $\norm{\mathcal{X}} + \omega' \geq \frac{1}{2}\brac{\brac{1+\frac{\gamma}{3}}^{n}\alpha+1}\omega - \oh{\omega}$ colors in total.
Easy calculation shows that for $\gamma \in \brac{0,1}$, $\alpha \in \sbrac{1,3}$
and for any $ n \geq \log_{1+\frac{\gamma}{3}}\brac{5/2-\gamma}\text{,}$
we have $\frac{1}{2}+\frac{1}{2}\brac{1+\frac{\gamma}{3}}^{n}\alpha \geq \frac{5}{4}+\frac{1}{2}\brac{1-\gamma}\alpha$.

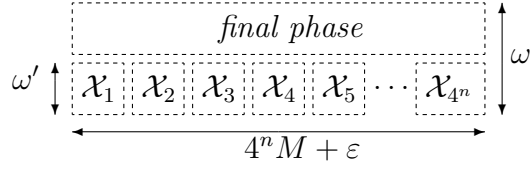
\begin{figure}
\centering
\setlength{\unitlength}{0.09in}
\begin{picture}(30,9)

\put(3.5,2.5){\dashbox{0.25}(3,3)}
\put(4,3.5){$\mathcal{X}_{1}$}
\put(7,2.5){\dashbox{0.25}(3,3)}
\put(7.5,3.5){$\mathcal{X}_{2}$}
\put(10.5,2.5){\dashbox{0.25}(3,3)}
\put(11,3.5){$\mathcal{X}_{3}$}
\put(14,2.5){\dashbox{0.25}(3,3)}
\put(14.5,3.5){$\mathcal{X}_{4}$}
\put(17.5,2.5){\dashbox{0.25}(3,3)}
\put(18,3.5){$\mathcal{X}_{5}$}

\put(21,4){$\ldots$}

\put(23.5,2.5){\dashbox{0.25}(4,3)}
\put(24,3.5){$\mathcal{X}_{4^n}$}

\put(3.5,6){\dashbox{0.25}(24,3)}
\put(12,7){\emph{final phase}}

\put(15.5,1.5){\vector(-1,0){12}}
\put(15.5,1.5){\vector(1,0){12}}
\put(13.5,0){$4^{n}M+\epsi$}

\put(2.5,4.5){\vector(0,1){1}}
\put(2.5,4.5){\vector(0,-1){2}}
\put(0,4){$\omega'$}

\put(28.5,5.5){\vector(0,1){3.5}}
\put(28.5,5.5){\vector(0,-1){3}}
\put(29,5.5){$\omega$}

\end{picture}
\caption{Case 1: $\norm{\mathcal{X}}$ is large}
\label{fig:lower_52_case1}
\end{figure}

\subsubsection*{Case 2.}
The size of the intersection $\norm{\mathcal{Y}_1 \cap \ldots \cap \mathcal{Y}_4}$ is at least $\brac{1-\gamma} \cdot C$.
Let $\mathcal{Y} = \mathcal{Y}_1 \cap \mathcal{Y}_2 \cap \mathcal{Y}_3 \cap \mathcal{Y}_4$ denote the colors that appear in all four parts of the initial phase.
Presenter introduces a set $Z_1$ of $\omega'$ identical intervals covering all intervals contributing to $\mathcal{Y}_1$ and disjoint with intervals contributing to $\mathcal{Y}_2$.
See Figure~\ref{fig:lower_52_case2a}.
Let $\mathcal{Z}_1$ be the set of colors used by Algorithm to color $Z_1$.

Then Presenter introduces a set $Z_2$ of $\omega'$ identical intervals covering all intervals contributing to $\mathcal{Y}_4$ and disjoint with intervals contributing to $\mathcal{Y}_3$.
Let $\mathcal{Z}_2$ be the set of colors used by Algorithm to color $Z_2$.

Clearly, $\norm{\mathcal{Z}_1} = \norm{\mathcal{Z}_2} = \omega'$, and $\mathcal{Z}_1 \cap \mathcal{Y} = \mathcal{Z}_2 \cap \mathcal{Y} = \emptyset$.
Now we distinguish two subcases depending on the size of the set $\mathcal{Z}_2 \setminus \mathcal{Z}_1$.

\begin{figure}
\centering
\setlength{\unitlength}{0.09in}
\begin{picture}(51,16)

\put(3,2.5){\dashbox{0.25}(10,4)}
\put(7,4){$\mathcal{Y}_1$}

\put(15,2.5){\dashbox{0.25}(10,4)}
\put(19,4){$\mathcal{Y}_2$}

\put(27,2.5){\dashbox{0.25}(10,4)}
\put(31,4){$\mathcal{Y}_3$}

\put(39,2.5){\dashbox{0.25}(10,4)}
\put(43,4){$\mathcal{Y}_4$}

\put(3,7){\dashbox{0.25}(11.5,4)}
\put(8,8.5){$\mathcal{Z}_1$}

\put(37.5,7){\dashbox{0.25}(11.5,4)}
\put(43,8.5){$\mathcal{Z}_2$}

\put(13.5,11.5){\dashbox{0.25}(25,4)}
\put(25,13){$\mathcal{W}$}


\put(26,1.5){\vector(-1,0){23}}
\put(26,1.5){\vector(1,0){23}}
\put(23,0){$4^{n}M+\epsi$}


\put(1.5,7){\vector(0,-1){4.5}}
\put(1.5,7){\vector(0,1){4}}
\put(0,6.5){$\omega$}

\put(50,4.5){\vector(0,-1){2}}
\put(50,4.5){\vector(0,1){2}}
\put(50.5,4.5){$\omega'$}

\put(39.5,13.5){\vector(0,-1){2}}
\put(39.5,13.5){\vector(0,1){2}}
\put(40,13.5){$\omega'$}

\end{picture}
\caption{Case 2.1: $\norm{\mathcal{Y}}$ is large and $\norm{\mathcal{Z}_2 \setminus \mathcal{Z}_1} \geq \frac{1}{4}\omega$}
\label{fig:lower_52_case2a}
\end{figure}

\subsubsection*{Case 2.1.}
If $\norm{\mathcal{Z}_{2} \setminus \mathcal{Z}_{1}} \geq \frac{1}{4}\omega$, then Presenter introduces a set $W$ of $\omega'$ identical intervals intersecting all the intervals in $Z_1$ and $Z_2$, and covering all the intervals contributing to $\mathcal{Y}_2$ and $\mathcal{Y}_3$.
Let $\mathcal{W}$ be the set of colors used by Algorithm to color $W$. 
By the definition, we have $\mathcal{W} \cap \mathcal{Y} = \mathcal{W} \cap \mathcal{Z}_1 = \mathcal{W} \cap \mathcal{Z}_2 = \emptyset$.
Algorithm was forced to use $\norm{\mathcal{W}} + \norm{\mathcal{Z}_1} + \norm{\mathcal{Z}_2 \setminus \mathcal{Z}_1} + \norm{\mathcal{Y}} \geq 
\brac{\frac{1}{2}+\frac{1}{2}+\frac{1}{4}}\omega + \frac{1}{2}\brac{1-\gamma}\alpha\omega - \oh{\omega} = \brac{\frac{5}{4} + \frac{1}{2}\brac{1-\gamma}\alpha}\omega - \oh{\omega}$ colors in total.
See Figure~\ref{fig:lower_52_case2a}.

\begin{figure}
\centering
\setlength{\unitlength}{0.09in}
\begin{picture}(51,20)

\put(3,2.5){\dashbox{0.25}(10,4)}
\put(7,4){$\mathcal{Y}_1$}

\put(15,2.5){\dashbox{0.25}(10,4)}
\put(19,4){$\mathcal{Y}_2$}

\put(27,2.5){\dashbox{0.25}(10,4)}
\put(31,4){$\mathcal{Y}_3$}

\put(39,2.5){\dashbox{0.25}(10,4)}
\put(43,4){$\mathcal{Y}_4$}

\put(3,7){\dashbox{0.25}(11.5,4)}
\put(8,8.5){$\mathcal{Z}_1$}

\put(37.5,7){\dashbox{0.25}(11.5,4)}
\put(43,8.5){$\mathcal{Z}_2$}

\put(25.5,11.5){\dashbox{0.25}(13,4)}
\put(31,13){$W_2$}

\put(13,16){\dashbox{0.25}(13,4)}
\put(19,17.5){$W_1$}


\put(26,1.5){\vector(-1,0){23}}
\put(26,1.5){\vector(1,0){23}}
\put(23,0){$4^{n}M+\epsi$}


\put(1.5,7){\vector(0,-1){4.5}}
\put(1.5,7){\vector(0,1){4}}
\put(0,6.5){$\omega$}

\put(50,4.5){\vector(0,-1){2}}
\put(50,4.5){\vector(0,1){2}}
\put(50.5,4.5){$\omega'$}

\put(39.5,13.5){\vector(0,-1){2}}
\put(39.5,13.5){\vector(0,1){2}}
\put(40,13.5){$\omega'$}

\put(12,18){\vector(0,-1){2}}
\put(12,18){\vector(0,1){2}}
\put(9.5,18){$\omega'$}

\end{picture}
\caption{Case 2.2: $\norm{\mathcal{Y}}$ is large and $\norm{\mathcal{Z}_2 \setminus \mathcal{Z}_1} < \frac{1}{4}\omega$}
\label{fig:lower_52_case2b}
\end{figure}

\subsubsection*{Case 2.2.}
If $\norm{\mathcal{Z}_2 \setminus \mathcal{Z}_1} < \frac{1}{4}\omega$, then let $\mathcal{Z} = \mathcal{Z}_1 \cap \mathcal{Z}_2$ and observe that $\norm{\mathcal{Z}} \geq \floor{\frac{\omega}{4}}$.
Presenter introduces a set $W_1$ of $\omega'$ identical intervals intersecting all the intervals in $Z_1$, and covering all the intervals contributing to $\mathcal{Y}_2$. Then, presenter introduces a set $W_2$ of $\omega'$ identical intervals, intersecting all the intervals in $W_1$ and $Z_2$, and covering all the intervals contributing to $\mathcal{Y}_3$.
Let $\mathcal{W}$ be the set of colors used by Algorithm to color intervals in $W_1 \cup W_2$.
We have that $\norm{\mathcal{W}} = 2\omega'$, and $\mathcal{W} \cap \mathcal{Y} = \mathcal{W} \cap \mathcal{Z} = \emptyset$.
Algorithm was forced to use $\norm{\mathcal{W}} + \norm{\mathcal{Z}} + \norm{\mathcal{Y}} \geq \brac{1+\frac{1}{4}}\omega + \frac{1}{2}\brac{1-\gamma}\alpha\omega - \oh{\omega} = \brac{\frac{5}{4} + \frac{1}{2}\brac{1-\gamma}\alpha}\omega - \oh{\omega}$ colors in total.
See Figure~\ref{fig:lower_52_case2b}.
\end{proof}

\begin{corollary}\label{cor:lower_52}
There is an an $\tbrac{\alpha_{n}, 4^{n f\brac{\gamma}}+\epsi, 4^{n f\brac{\gamma}}+\epsi}$-schema,
for every $n \in \PosNat$, every $\epsi > 0$, and every $\gamma \in \brac{0,1}$, where
\[
\alpha_{n} = \frac{5}{2}\cdot\frac{1}{1+\gamma} - \frac{\brac{3-2\gamma}}{2\brac{1+\gamma}}\cdot\brac{\frac{1-\gamma}{2}}^{n}
\text{,}\quad
f\brac{\gamma} = \ceil{ \frac{ \log\brac{\frac{5}{2}-\gamma} }{ \log\brac{1+\frac{\gamma}{3}}} }
\text{.}
\]

\end{corollary}

\begin{proof}
The argument is similar to Corollaries~\ref{cor:lower_53}, and~\ref{cor:lower_74}, but now we  solve the recurrence equations
$\alpha_{0} = 1$, $\alpha_{n+1} = \frac{5}{4} + \frac{1}{2}\brac{1-\gamma}\alpha_{n}$ for competitive ratio,
and $M_{0} = 1$, $M_{n+1} = 4^{f\brac{\gamma}}M_{n}$, $\sigma_{n} = M_{n}$ for region and interval lengths.
\end{proof}

\FiveHalvesThm*

\begin{proof}
Assume for contradiction that, for some $\epsi > 0$, there are $(5/2 - \epsi)$-competitive algorithms for every $\sigma \geq 1$. Setting $\gamma$ small enough and $n$ large enough, Corollary~\ref{cor:lower_52} gives us a $\tbrac{\frac{5}{2} - \frac{\epsi}{4}, \sigma, \sigma}$-schema, for some value of $\sigma$. This means, there is $\omega_P$ such that for every $\omega \geq \omega_P$ there exists an $\tbrac{\omega, (\frac{5}{2} - \frac{2\epsi}{4})\omega, \sigma, \sigma }$-strategy. On the other hand, for the assumed $\sigma$-interval coloring algorithm, there exists $\omega_A$ such that for every $\omega \geq \omega_A$ the algorithm uses at most $\brac{\frac{5}{2}-\frac{3\epsi}{4}}\omega$ colors for every $\omega$-colorable set of intervals. For $\omega=\max(\omega_A,\omega_P)$ we reach a contradiction.
\end{proof}

\bibliographystyle{plainurl}
\bibliography{paper}

\clearpage
\appendix

\section{Open Problems}

Throughout the paper we presented several constructions which can be combined recursively to obtain strategies proving higher and higher lower bounds for larger and larger interval lengths $\sigma$. Table~\ref{tab:strategies} summarizes a selection of these strategies.

\begin{table}
\centering
\caption{Summary of selected strategies for Presenter}
\label{tab:strategies}
\begin{tabular}{@{}lll@{}}
\toprule
ratio & interval length & strategy \\
\midrule
$1.5$  & $1$       & Epstein and Levy~\cite{EpsteinL05} \\
$1.6$  & $2+\epsi$ & Corollary~\ref{cor:informal_lower_32}, $n=2$ iterations \\
\addlinespace
$1.66$ & $1+\epsi$ & Corollary~\ref{cor:lower_53}, $n=1$ iteration \\
$1.72$ & $3+\epsi$ & Corollary~\ref{cor:lower_53}, $n=2$ iterations \\
\addlinespace
$1.75$ & $2+\epsi$ & Corollary~\ref{cor:lower_74}, $n=1$ iteration \\
$1.81$ & $8+\epsi$ & Corollary~\ref{cor:lower_74}, $n=2$ iterations \\
\addlinespace
$2$    & $4^{39}+\epsi$   & Corollary~\ref{cor:lower_52}, $n=3$ iterations, $\gamma=0.21030395$\\
$2.4$  & $4^{486}+\epsi$  & Corollary~\ref{cor:lower_52}, $n=6$ iterations, $\gamma=0.0339$\\
$2.49$ & $4^{7970}+\epsi$ & Corollary~\ref{cor:lower_52}, $n=10$ iterations, $\gamma=0.003449$\\
\bottomrule
\end{tabular}
\end{table}

There are still large gaps between the best known lower and upper bounds for
the optimal competitive ratios for online $\sigma$-interval coloring problems
(see Figure~\ref{fig:gap}).
It would be interesting to close the gap, even for a single specific $\sigma$.
For example, for $\sigma=3/2$ the optimal online algorithm has the
competitive ratio somewhere between $5/3$ and $5/2$.

Finally, let us conjecture that the lower bound of Theorem~\ref{thm:52} is tight.

\begin{conjecture}\label{cnj:upper_52}
There is a $5/2$-competitive online algorithm for $\sigma$-interval coloring,
for every $\sigma \ge 1$.
\end{conjecture}

\begin{figure}[h]
\centering
\includegraphics{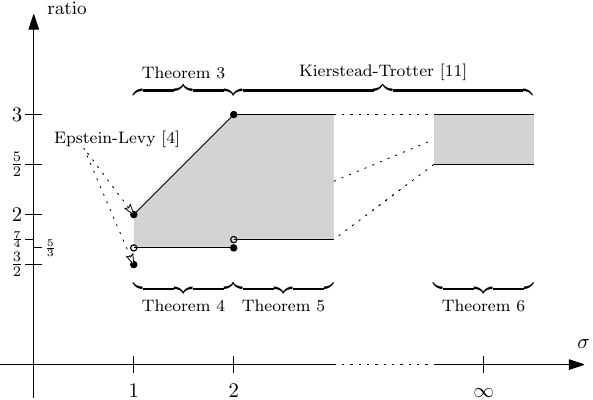}
\caption{Gap between current bounds for competitive ratio of online $\sigma$-interval coloring}
\label{fig:gap}
\end{figure}


\end{document}